\documentclass[12pt]{article}
\usepackage[margin=25mm]{geometry}                
\usepackage{graphicx}
\usepackage{amssymb,amsmath,amsthm,color,tikz,url,soul} 
\usetikzlibrary{calc}
\newtheorem{theorem}{Theorem}
\newtheorem{lemma}{Lemma}
\theoremstyle{remark}
\newtheorem*{remark}{Remark}

\newtheorem{example}{Example}
\DeclareMathOperator{\KP}{\textit{K}\,}

\DeclareMathOperator{\ttt}{\mathbf{t}\mskip 0.5mu}
\DeclareMathOperator{\dd}{\mathbf{d}\mskip 0.5mu}
\newcommand{\cnd}{\mskip 1mu | \mskip 1mu}
\newcommand{\rd}[1]{{}}
\newcommand{\bl}[1]{{#1}}

\newcommand{\bbar}[1]{\overline{#1}\vphantom{#1}} 

\def\eps{\varepsilon}

\title{Conditional probabilities and\\ van Lambalgen theorem revisited\footnote{The paper records the discussion of conditional randomness results by the participants of the Heidelberg Focus Semester on algorithmic randomness, 2015 (Laurent Bienvenu, \bl{Mathieu Hoyrup, Rupert H\"olzl}, Wolfgang Merkle, Jason Rute and others)}}
\author{Bruno Bauwens\thanks{National Research University Higher School of Economics (HSE), Faculty of Computer Science, Kochnovskiy Proezd 3, Moscow, 125319 Russia, \protect\url{bbauwens@hse.ru}} \and Alexander Shen\thanks{Laboratoire d'Informatique, de Robotique et de Micro\'electronique de Montpellier, CNRS, Universit\'e de Montpellier, \protect\url{alexander.shen@lirmm.fr}. Supported by ANR RaCAF grant.} \and Hayato Takahashi\thanks{1-1 Yanagido, Gifu City 501-1193, Japan. Organization for Promotion of Higher Education and Student Support, Gifu University. \protect\url{hayato.takahashi@ieee.org}. Supported by JSPS KAKENHI grant number 24540153.}}

\date{}

\begin{document}

\maketitle

\begin{abstract}
The definition of conditional probability in case of continuous distributions  was an important step in the development of mathematical theory of probabilities. How can we define this notion in algorithmic probability theory? In this survey we discuss the developments in this direction trying to explain what are the difficulties and what can be done to avoid them. Almost all the results discussed in this paper have been published (and we provide the references), but we tried to put them into perspective and to explain the proofs in a more intuitive way. We assume that the reader is familiar with basic notions of measure theory and algorithmic randomness (see, e.g.,~\cite{usv} or~\cite{ShenSurvey} for a short introduction).
\end{abstract}

\section{Conditional probability}

Let $P$ be a computable measure on the product of the two copies of the Cantor space $\Omega_1\times\Omega_2$, and let $P_1$ be its marginal distribution (=projection of $P$ onto $\Omega_1$).  Consider some $\alpha_1\in\Omega_1$. We want to define the conditional distribution on $\Omega_2$ with the condition ``the first coordinate is equal to $\alpha_1$''. For that we consider a prefix $a_1$ of $\alpha_1$ and the conditional distribution on $\Omega_2$ with the condition ``the first coordinate starts with $a_1$''. (For this we need that $P_1(a_1)$, the probability of the interval $[a_1]$ of all extensions of $a_1$, to be positive.) In this way we get a family of measures $P_{a_1}$ on $\Omega_2$:
$$
P_{a_1}(a_2)=\frac{P(a_1,a_2)}{P_1(a_1)}
$$
Here the numerator is the $P$-measure of the product $[a_1]\times[a_2]$, and the denominator is the $P$-measure of $[a_1]\times \Omega_2$. Then, for given $a_2$, we consider the limit of probability $P_{a_1}(a_2)$ as the length of prefix $a_1$ (of $\alpha_1$) tends to infinity.

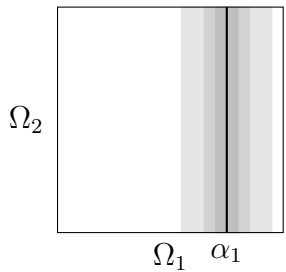
\begin{center}
\begin{tikzpicture}[scale=0.75]
      \begin{scope}[shift={(0.2,0)}]
	\foreach \x/\xx/\col in {2/3.6/gray!20,2.4/3.2/gray!35,2.6/3.0/gray!50}
	{
	  \filldraw[\col] (\x,0) rectangle (\xx,4);
	  }
	\draw[thick] (2.8,0) node[anchor=north] {$\alpha_1$} -- (2.8,4);
      \end{scope}
         \draw (0,0) -- node[anchor=east] {$\Omega_2$} (0,4) 
      -- (4,4) -- (4,0) -- node[anchor=north] {$\Omega_1$} (0,0) ;
\end{tikzpicture}
\end{center}

\begin{theorem}[\cite{TakahashiDefinition}]
If $\alpha_1$ is Martin-L\"of random with respect to $P_1$, then this limit is well defined and determines a measure on $\Omega_2$.
\end{theorem}

\begin{proof}
For a fixed $a_2$ the function $m\colon a_1\mapsto P_{a_1}(a_2)$ is a computable martingale on $\Omega_1$ with respect to $P_1$ (being a ratio of some measure and $P_1$; note that the denominator $P_1(a_1)$ is not zero since $\alpha_1$ is random with respect to $P_1$), so the limit exists due to effective martingale convergence theorem. Here is its proof. If $m$ is a computable non-negative martingale, then for every rational $c$ the set of sequences along which $m$ exceeds $c$ somewhere, is effectively open, and has measure at most $1/c$, if the initial capital is~$1$. So $m$ is bounded along every Martin-L\"of random sequence. Also, for every pair of positive rationals $u<v$, the set of sequences where the martingale is infinitely often less than $u$ and infinitely often greater than $v$, is an effectively null set. Indeed, the set of sequences where there are at least $N$ changes across $(u,v)$, is effectively open, and its measure is small due to ``buy low --- sell high'' argument. Formally, we consider another martingale that follows $m$'s bets, starting when $m$ becomes less than $u$, until $m$ becomes greater than $v$, and then waits for the next time when $m$ becomes less than $u$, starts following $m$ again, etc. So for Martin-L\"of sequences convergence is guaranteed.\footnote{Technically the martingale $m$ may be infinite if $P_1(a_1)=0$ for some strings $a_1$, but this does not matter for the argument: these intervals are covered by an effectively null set we construct.}

To see that we indeed get a measure on $\Omega_2$ in this way, we have to check the (finite) additivity which is obvious (limit of the sum is the sum of limits).
\end{proof}

This measure can be denoted by $P(\cdot \cnd\alpha_1)$. Of course, if we start with a product measure $P=P_1\times P_2$, the conditional probability is the same for all $\alpha_1$, and equals $P_2$.

\begin{remark}
The definition of a conditional measure in the classical probability theory is usually given using Radon--Nikodym theorem; then the Lebesgue differentiation theorem can be used to show that the conditional measure defined in this way coincides almost everywhere with the limit we considered.
\end{remark}

\section{Non-computable conditional probability}

Let us note first that the limit in the definition of conditional probably may not exist for some conditions (though these conditions form a null set, as we have seen). This is shown by the following

\begin{example}
Consider the following distribution on $\Omega_1\times\Omega_2$. We identify this product with the square $[0,1]\times[0,1]$; the binary-rational points have two representations as sequences, but this does not matter much. In our distribution the grey areas have double density compared with the uniform distribution on the square, while the white areas have zero density.
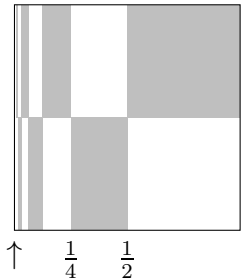
\begin{center}
 \begin{tikzpicture}[scale=0.75]
    
      \foreach \x/\y in {2/2,1/0,0.5/2,0.25/0,0.125/2,0.062/0,0.03/2}
        \fill[gray!50] (\x,\y) rectangle +(\x,2);

      \draw (0,0) rectangle (4,4);
      \node[anchor=north] {$\uparrow$};
      \node[anchor=north] at (2,0) {$\frac{1}{2}$};
      \node[anchor=north] at (1,0) {$\frac{1}{4}$};
    
\end{tikzpicture}
\end{center}
Then for the leftmost point (shown by an arrow) the limit distribution does not exist. Indeed, the conditional probabilities of two halves oscillate between $1/3$ and $2/3$ depending on the length of $a_1$, as one may easily check.
\end{example}

\bl{Our next example shows that the conditional probability for a computable distribution on pairs might exist but be non-computable.  A first example of this type was constructed in~\cite{afr}.\footnote{In fact, the example in~\cite{afr} has an additional property: the set of $\alpha_1$ for which $P(\cdot\cnd \alpha_1)$ is not computable with oracle $\alpha_1$, has positive $P_1$-measure. The measure constructed in Example~\ref{ex:bauwens} (see below) also has this property, see~\cite[Corollary 2]{counterExample}. On the other hand, the example in~\cite{afr} has a conditional measure $P(\cdot\cnd\alpha_1)$ that is continuous in $\alpha_1$, unlike Example~\ref{ex:bauwens}.
}}
\rd{Edited the footnote, please check!}

\begin{example}
Let $a_1,a_2,\ldots$ be an increasing computable sequence of rational numbers whose limit $\alpha$ is non-computable. Consider the following distribution:

\smallskip

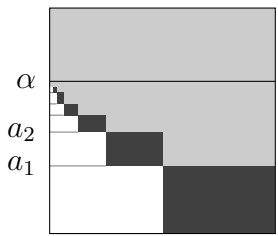
\begin{center}
 \begin{tikzpicture}[scale=0.75]
        \fill[gray!40] (0,0) rectangle (4,4);
	\foreach \x/\y/\yy in {2/0/1.2,1/1.2/1.8,0.5/1.8/2.1,0.25/2.1/2.3,0.125/2.3/2.5,0.063/2.5/2.6}
	{
	  \fill[darkgray] (\x,\y) rectangle ($(\x,0) + (\x,\yy)$);
	  \fill[white] (\x,\yy) rectangle (0,0);
	  }
	\foreach \x/\y/\yy in {2/0/1.2,1/1.2/1.8,0.5/1.8/2.1,0.25/2.1/2.3,0.125/2.3/2.5,0.063/2.5/2.6}
	{
	  \draw[gray] (0,\y) -- (\x,\y);
	  }
	  \node[anchor=east] at (0,1.2) {$a_1$};
	  \node[anchor=east] at (0,1.8) {$a_2$};

	  \draw[] (0,2.7) node[anchor=east] {$\alpha$} -- +(4,0);

	\draw (0,0) rectangle (4,4);
\end{tikzpicture}
\end{center}
\medskip
Vertical lines are drawn at points $1/2$, $1/4$,\ldots; in the grey zone the density is the same as for the uniform distribution; in the black zone the density is twice bigger, and in the white zone the density is zero. Since the widths of the black and white stripes on every horizontal line are the same,  the total amount of mass does not change; we just move all the mass horizontally from the white part to the black part.

Note that the distribution on the square is computable even though $\alpha$ is not computable. Indeed, if we are interested in the mass of some rectangle $R$ (the product of two binary intervals), the mass transfers in the small rectangles (thinner than $R$) do not matter, and we may look only on finitely many $a_i$ (they can be computed).

It is easy to see that the limit distribution (at the leftmost point) is the uniform distrubition on $[\alpha,1]$, and it is not computable, since the density $1/(1-\alpha)$ is not computable.
\end{example}

In this example the conditional distribution is non-computable only at one point. However, the example can be easily changed so that the conditional distribution is the same non-computable distribution for all $\alpha_1$ except for binary-rational numbers (sequences with finitely many zeros or ones).
\begin{example}\label{ex:bauwens}
Consider again the increasing computable sequence $a_1,a_2,\ldots$ of rational numbers that converges to a non-computable real $\alpha$.
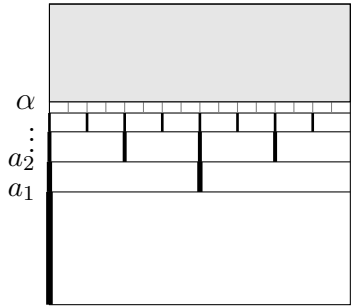
\medskip
\begin{center}
     \begin{tikzpicture}
    \def\a{1.5};
    \def\aa{1.9};
    \def\aaa{2.3};
    \def\afour{2.55};
    \def\alim{2.7};
    \draw[black] (0,\a)   node[anchor=east]  {\small{$a_1$}}  -- +(4,0);
    \draw (0,\aa)  node[anchor=east] {\small{$a_2$}}  -- +(4,0);
    \draw (0,\aaa) node[anchor=east] {\small{$\vdots$}}  -- +(4,0);
    \draw (0,\afour)   -- +(4,0);
    \draw       (0,\alim)  node[anchor=east] {$\alpha$}    -- +(4,0);
    \draw[line width=2.5pt] (0.00,0) -- +(0,\a);
    \foreach \x  in {0.00,2.00}
      {\draw[line width=2pt] (\x,\a) -- (\x,\aa);}
    \foreach \x  in {0.0,1,...,3} 
      {\draw[line width=1.5pt] (\x,\aa) -- (\x,\aaa);}
    \foreach \x  in {0.00,0.50,...,3.50} {
      \draw[line width=1pt] (\x,\aaa) -- (\x,\afour);
      }
    \foreach \x  in {0,0.25,...,3.75} 
      {\draw[gray] (\x,\afour) -- (\x,\alim);}
    \filldraw[fill=gray!20] (0,\alim) rectangle (4,4);
    \draw rectangle (4,4);
  \end{tikzpicture}
\end{center}
\end{example}
\medskip
In the grey area above the horizontal line with coordinate $\alpha$ we still keep the density the same as in the uniform distribution; however, below $\alpha$ all the mass is concentrated on black vertical segments. For example, the mass $a_1$ is concentrated on the segment $\{0\}\times [0,a_1]$, and is distributed uniformly there (so the mass transfer happens only in the horizontal direction). The mass $a_2-a_1$ is then split evenly between two vertical segments shown (at horizontal coordinates $0$ and $1/2$), etc. One can say that each vertical segment ``horizontally grabs'' all the mass of the white rectangle on the right of it (so the latter has zero density except for its left side).

As before, it is easy to see that the resulting distribution is computable: to find the mass of a binary rectangle of width $2^{-n}$, it is enough to take into account only $a_1,\ldots,a_n$ and use the uniform distribution above $a_n$.

It is also easy to see that for every $\alpha_1$ that is not binary rational (in other words, for vertical lines that do not cross black segments), the conditional probability $P(\cdot\cnd \alpha_1)$ is uniformly distributed on $[\alpha,1]$.

We will reuse this example in Section~\ref{sec:bauwens-example}. 

\section{van Lambalgen theorem}

Now we consider the relation between the randomness of the pair and its components. The basic result in this field goes back to Michiel val Lambalgen (see~\cite[Theorem 5.10]{mvl}, where this result is stated in an implicit way). It considers the case of the product $P$ of two computable measures $P_1$ on $\Omega_1$ and $P_2$ on $\Omega_2$, and says that the pair $(\alpha_1,\alpha_2)$ is Martin-L\"of random with respect to $P$ if and only if two conditions are satisfied: 
\begin{itemize}
\item $\alpha_1$ is random with respect to the measure $P_1$; 
\item $\alpha_2$ is random with respect to the measure $P_2$ with oracle~$\alpha_1$. 
\end{itemize}
(See, for example, \cite[chapter 5]{usv} for the proof.) Note that these conditions are not symmetric; of course, one can exchange the coordinates and conclude that $\alpha_1$ is also random with respect to $P_1$ with oracle~$\alpha_2$.

It is natural to look for some version of van Lambalgen theorem generalized to non-product measures $P$. Informally speaking, such a version should say that $(\alpha_1,\alpha_2)$ is $P$-random if and only if 
\begin{itemize}
\item $\alpha_1$ is $P_1$-random (where $P_1$ is the projection of $P$);
\item $\alpha_2$ is random with respect to the conditional probability measure $P(\cdot \cnd\alpha_1)$ with oracle~$\alpha_1$. 
\end{itemize}
But some precautions are needed. The problem is that Martin-L\"of randomness is  usually defined for computable measures, while the conditional measure is defined as a limit. As we have seen, it may not be computable even with oracle~$\alpha_1$ for $P_1$-random $\alpha_1$. Indeed, in our example the conditional distribution was $[\alpha,1]$ for every irrational condition $\alpha_1$, and among them there are $P_1$-random conditions that do not compute~$\alpha$. To prove this, let us note that uniformly random reals are all $P_1$-random and some of them do not compute~$\alpha_1$.  Indeed, if the uniform measure of conditions $\alpha_1$ that compute $\alpha$ were positive, then the same would be true for some fixed oracle machine due to countable additivity of the uniform measure. Then Lebesgue density theorem says that there is some interval where most of the oracles compute $\alpha$, so $\alpha$ can be computed without oracle by majority voting --- but $\alpha$ is not computable. (The last argument is known as de Leeuw -- Moore -- Shannon -- Shapiro theorem~\cite{dlmss}.)

Still we can make several observations.

\section{Image randomness and beyond}

\emph{If $(\alpha_1,\alpha_2)$ is \textup(Martin-L\"of\textup) random with respect to $P$, then $\alpha_1$ is \textup(Martin-L\"of\textup) random with respect to the marginal distribution~$P_1$}. This is obvious; every cover of $\alpha_1$ with small $P_1$-measure gives a cover of $(\alpha_1,\alpha_2)$ with the same $P$-measure, being multiplied by $\Omega_2$. 

This result can be considered also as a special case of the image randomness theorem (see the section about image randomness in~\cite{usv}) applied to the projection mapping.  Moreover, the reverse direction of image randomness theorem (``no randomness from nothing'') guarantees that \emph{every $P_1$-random $\alpha_1$ is a first component of some $P$-random pair $(\alpha_1,\alpha_2)$}.

So we know that for every $P_1$-random $\alpha_1$ there exists at least one $\alpha_2$ that makes the pair $(\alpha_1,\alpha_2)$ $P$-random. It is natural to expect that there are much more. Indeed this is the case, as the following result from~\cite{TakahashiDefinition} shows.

\begin{theorem}
Let $\alpha_1$ be $P_1$-random. Then the set of $\alpha_2$ such that $(\alpha_1,\alpha_2)$ is $P$-random, has probability $1$ according to the conditional probability distribution $P(\cdot\cnd \alpha_1)$.
\end{theorem}

\begin{proof}
Consider a universal Martin-L\"of test on the product space. Let $U_n$ be the effectively open set of measure at most $2^{-2n}$ provided by this test. In the case of product measure $P_1\times P_2$ we would consider the set $V_n$ of all $\alpha_1$ such that the $\alpha_1$-section of $U_n$ has $P_2$-measure greater than $2^{-n}$, note that $V_n$ is an effectively open set of measure at most $2^{-n}$, and conclude that random $\alpha_1$ do not belong to $V_n$ for all sufficiently large~$n$. (If a point $x$ is covered by infinitely many sets $V_n$, it is covered by all sets $\overline V\!_n=\bigcup_{k>n} V_k$, so $x$ is not random. This is often called the \emph{Solovay randomness criterion}.)  So for every random $\alpha_1$ and for sufficiently large values of $n$ the $\alpha_1$-section of $U_n$ has measure at most $2^{-n}$, so almost every $\alpha_2$ lies outside the $\alpha_1$-section of $U_n$ for all large $n$, which gives the desired result. Moreover, every $\alpha_2$ that is $P_2$-random with oracle $\alpha_1$ works, since the open cover for $\alpha_2$ provided by the $\alpha_1$-section of $U_n$ is $\alpha_1$-enumerable. (This is how the van Lambalgen theorem is proven.)

For the general case of non-product measure $P$ we should be more careful since the conditional probability is defined only in the limit. Instead of $V_n$, we consider all $n$-\emph{heavy} intervals $I$ in $\Omega_1$, i.e., all intervals $I$ such that $U_n$ occupies more than $2^{-n}$-fraction in $I\times \Omega_2$ measured according to $P$, in other words, all intervals $I$ such that $P(U_n\cap (I\times\Omega_2))> 2^{-n}P_1(I)$. This is an enumerable family of intervals since $P$ and $P_1$ are computable.

\begin{lemma}\label{lem:heavy-intervals}
The $P_1$-measure of the union of all $n$-heavy intervals is at most $2^{-n}$.
\end{lemma}

\begin{proof}
To prove that the union of $n$-heavy intervals has measure at most $2^{-n}$, it is enough to prove this for every finite  union of $n$-heavy intervals.  Without loss of generality we may assume that intervals in this union are disjoint (consider only maximal intervals). For every $n$-heavy interval $I$ the fraction of $U_n$ in the stripe $I\times \Omega$ exceeds $2^{-n}$, so the total $P_1$-measure of disjoint $n$-heavy intervals cannot exceed $2^{-n}$, otherwise $U_n$ would be too big (its measure would be greater than $2^{-2n}$).
\end{proof}

In other words, the function $$I\mapsto \text{fraction of $U_n$ in $I\times\Omega_2$}$$ is a (lower semicomputable) martingale with initial value $2^{-2n}$. So due to the martingale inequality the union of intervals where the martingale exceeds $2^{-n}$ is at most $2^{-n}$.

\medskip

\begin{lemma}\label{lem:good-alpha}
If $\alpha_1$ is outside any $n$-heavy interval, then the $\alpha_1$-section of $U_n$ has measure at most $2^{-n}$ according to the conditional probability with condition $\alpha_1$.
\end{lemma} 

\begin{proof}
If the conditional measure of the $\alpha_1$-section of $U_n$ exceeds $2^{-n}$, then there exists a finite set of disjoint vertical intervals $J_1,\ldots,J_k$ that have total conditional measure more than $2^{-n}$ and all belong to the $\alpha_1$-section of $U_n$. Since $U_n$ is open, the compactness argument shows that for sufficiently small intervals $I$ containing $\alpha_1$ we have
$$
I\times J_1,\ldots,I\times J_k \subset U_n.
$$
By assumption, the conditional measure of $J_1\cup\ldots\cup J_n$ exceeds $2^{-n}$, and the conditional probability is defined as the limit of conditional probabilities with condition $I$ when intervals $I$ containing $\alpha_1$ decrease. So for all sufficiently small $I$ the conditional measure of $J_1\cup\ldots\cup J_k$ with condition $I$ exceeds $2^{-n}$, but this means that $I$ is $n$-heavy and $\alpha_1$ is covered by $I$, contrary to our assumption. 
\end{proof}

Now consider (for some fixed random $\alpha_1$) all  $\alpha_2$ such that $(\alpha_1,\alpha_2)$ is non-random. Being non-random, these pairs belongs to all $U_n$, so for such a pair $\alpha_2$ is inside $\alpha_1$-sections of $U_n$ for all~$n$.  Since $\alpha_1$ is random, it is not covered by $n$-heavy  intervals $I$ for all sufficiently large~$n$ and by Lemma~\ref{lem:good-alpha} all bad $\alpha_2$ are covered by a set of conditional measure $2^{-n}$ at most $n$ (for all large~$n$). So the $\alpha_1$-conditional measure of the set of bad $\alpha_2$ is equal to~$0$.
\end{proof}
 
In fact, we have proven the following result from~\cite{TakahashiProduct,TakahashiGeneral} (one direction of van Lambalgen theorem).
\begin{theorem}\label{th:easy-direction}
If $\alpha_1$ is $P_1$-random and $\alpha_2$ is blind \textup(Hippocratic\textup) random with respect to the conditional probability $P(\cdot\cnd\alpha_1)$, then the pair $(\alpha_1,\alpha_2)$ is $P$-random.
\end{theorem}

By \emph{blind} (\emph{Hippocratic}) randomness we mean a version of Martin-L\"of definition of randomness with respect to noncomputable measure. In this version (studied by Kjos-Hansen~\cite{kh}) uniformly effectively open tests are considered and the random sequence is required to pass all of them (if the measure is non-computable, there may be no universal test). It is opposed to \emph{uniform} randomness where the test is effectively open with respect to the measure (see~\cite{survey} for the details).

\begin{proof}[Proof of Theorem~\textup{\ref{th:easy-direction}}]
Indeed, in the construction above we get a cover for bad $\alpha_2$ that is enumerable with oracle $\alpha_1$.
\end{proof}

\section{A counterexample}\label{sec:bauwens-example}

The following counterexample from~\cite{counterExample} shows that the statement of Theorem~\ref{th:easy-direction} cannot be reversed.

\begin{theorem}
There exists a computable measure $P$ on $\Omega_1\times\Omega_2$, for which conditional measure $P(\cdot \cnd\alpha_1)$ is defined for all $\alpha_1$, 
and a $P$-random pair $(\alpha_1,\alpha_2)$ such
that $\alpha_2$ is not blind random with oracle $\alpha_1$ with respect to conditional distribution
$P(\cdot\cnd\alpha_1)$ on $\Omega_2$.  
\end{theorem}
\rd{Sorry, wrote $P(.|\alpha_2)$ now it is $P(.|\alpha_1)$. At the formal level, should one not require that conditional measure is defined before it is claimed not to be random relative to it? --- We claim that $(\alpha_1,\alpha_2)$ is a random pair, so $\alpha_1$ is random and conditional measure exists. Still additional precautions do not hurt.}

\begin{proof}
  We use the measure from Example~\ref{ex:bauwens}. The second component $\alpha_2$ of the pair is now a lower semicomputable random real $\alpha$ that is the limit of a computable increasing sequence of binary fractions $a_i$ (Chaitin's $\Omega$-number).  We start with the following observation: for this $\alpha_2$ \emph{the pair $(\alpha_1,\alpha_2)$ is random if and only if this pair is random with respect to the uniform measure}. Indeed, if we have some enumerable set of rectangles that covers $(\alpha_1,\alpha_2)$, we can safely discard parts of the rectangles that are below some $a_i$, since this does not change anything for $(\alpha_1,\alpha_2)$. In this way we may ensure that the $P$-measure of these rectangles equals their uniform measure (for a thin rectangle we need to discard more of it), so \bl{a} $P$-test can be transformed into a uniform test and vice versa if we are interested only in points with second coordinate $\alpha_2=\alpha$.

  Now we can find a random point $(\alpha_1,\alpha_2)$ with second coordinate $\alpha_2=\alpha$ (according to classical van Lambalgen theorem it is enough to take $\alpha_1$ that is random with respect to uniform measure with oracle $\alpha$). Since $\alpha_1$ is random and not binary-rational, the conditional probability $P(\cdot\cnd\alpha_1)$ is uniformly distributed on $[\alpha,1]$. It remains to show that the \emph{lower semicomputable real $\alpha$ is not blind random with respect to the uniform distribution on $[\alpha,1]$}. Indeed, for every rational $\eps>0$ the interval $(0,\alpha+\eps)$ is effectively open, since it can be represented as the union of $(0,a_i+\eps)$, and its measure with respect to the uniform measure on $[\alpha,1]$ is proportional to $\eps$ (so it is small for small~$\eps$).
\end{proof}

\section{The case of computable conditional measure}

Still the van Lambalgen result can be generalized to non-product measure with an additional computability assumption. As before, we consider a computable measure $P$ on $\Omega_1\times\Omega_2$ and its projection $P_1$ on~$\Omega_1$ (the marginal distribution). The following result was proven by Hayato Takahashi~\cite{TakahashiDefinition,TakahashiProduct}:

\begin{theorem}\label{th:difficult-direction}
If a pair $(\alpha_1,\alpha_2)$ is $P$-random, and the conditional distribution $P(\cdot\cnd\alpha_1)$ is computable with oracle $\alpha_1$, then $\alpha_2$ is Martin-L\"of random with oracle $\alpha_1$ with respect to this conditional distribution.
\end{theorem}

Before proving this theorem, let us make several remarks about its statement:
\begin{itemize}
\item 
Under our assumption Martin-L\"of randomness is well defined (the distribution is computable with oracle $\alpha_1$). 
\item
  We already know that $\alpha_1$ is $P$-random, so we get a randomness criterion for pairs (assuming the conditional distribution is computable given the condition). 
\item
We assume the computability of conditional distribution \emph{only for condition} $\alpha_1$; for other random elements of $\Omega_1$ (used as conditions) the conditional distribution may not be computable.
\end{itemize}

\begin{proof}[Proof of Theorem~\textup{\ref{th:difficult-direction}}]
Let us first recall the proof for the case of a product measure $P_1\times P_2$. Assume that $\alpha_2$ is \emph{not} random. Then there is a set $Z\subset\Omega_2$ of arbitrarily small $P_2$-measure that covers $\alpha_2$ and is \emph{effectively open with oracle} $\alpha_1$. The latter statement means that $Z$ is a section of some effectively open set of pairs $U\subset\Omega_1\times\Omega_2$ obtained by fixing the first coordinate equal to $\alpha_1$. This set $U$ covers $(\alpha_1,\alpha_2)$ by construction. The problem is that only the $\alpha_1$-section of $U$ is guaranteed to be small while other sections may be large, and we need a bound for the total measure of $U$ to show the non-randomness of $(\alpha_1,\alpha_2)$. 

The solution is that we ``trim'' $U$ making all its sections small. Enumerating the rectangles in $U$, we look at the $P_2$-size of all sections. When some section attempts to become too big, we prevent this and stop increasing that section. In this way we miss nothing in the $\alpha_1$-section of $U$ since it was small in the first place. 

This argument works for the case of product distributions. How can we do similar things in the general case of arbitrary computable measures on $\Omega_1\times \Omega_2$?  Again we start with a set $Z$ of small conditional measure containing $\alpha_2$ and represent $Z$ as $\alpha_1$-section of some effectively open $U\subset \Omega_1\times \Omega_2$. But trimming $U$ is now not so easy. To understand the problem better, let us first consider two simple approaches that do not work.

\smallskip

\emph{First non-working approach}. The problem for the general case is that we have no ``etalon'' measure on sections that can be used for trimming. It is natural to use the conditional measure $P(\cdot\cnd\alpha_1)$, and by assumption, there is an algorithm $\Gamma$ that computes it using $\alpha_1$ as oracle. Given some rectangle, we can split this rectangle horizontally (i.e., fix more and more bits of $\alpha_1$) and use $\Gamma$ to compute the conditional measure with more and more precision for all the parts, letting through only the rectangles where the vertical side is guaranteed to have small ($\Gamma$-computed) measure for the values of $\alpha_1$ that belong to the horizontal side.

The problem with this approach is that the conditional measure is computable given $\alpha_1$ \emph{only on the $\alpha_1$-section} but not elsewhere. So the algorithm $\Gamma$ may have no relation to the measure $P$ outside this section. In this case small values produced by $\Gamma$ do not guarantee anything about the measure of the rectangles that are let through, and we are in trouble.

\smallskip

\emph{Second non-working approach}. Instead of computing the conditional probability, we may use the actual conditional probability when the condition is an interval; unlike the limit probability, it is computable. Imagine that we have some rectangle $A\times B$. Then we can compute the conditional probability of $B$ with condition ``the first coordinate is in $A$'', i.e., the ratio $P(A\times B)/P_1(A)$, and let the rectangle through if this ratio is small (we assume that there are no earlier rectangles in the same vertical stripe). This guarantees that the $P$-measure of the rectangle is small; if we have several allowed rectangles with disjoint horizontal footprints, and for each of them this conditional probability is at most~$\eps$, then the $P$-measure of their union is also at most $\eps$, since for each of them the $P$-measure is bounded by $\eps$ times the horizontal size of the rectangle, and the sum of horizontal sizes is at most $1$.

What is the problem with this approach? (There should be a problem, since in this argument we do not use $\Gamma$, and this is unavoidable, as the counterexample above shows.) The problem becomes clear if we consider the case of overlapping rectangles. 

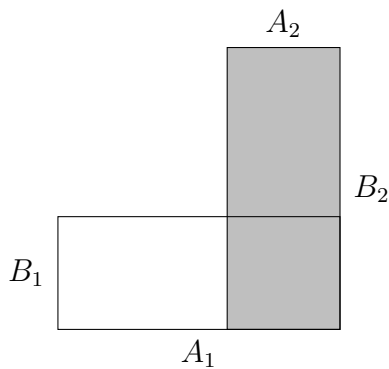
\begin{center}
 \begin{tikzpicture}[scale=0.75]
    
      \draw[fill=gray!50, draw=black] (3,0) rectangle +(2,5);
      \draw (0,0) rectangle +(5,2);

      \node[anchor=north] at (2.5,0) {$A_1$};
      \node[anchor=south] at (4,5) {$A_2$};
      \node[anchor=east] at (0,1) {$B_1$};
      \node[anchor=west] at (5,2.5) {$B_2$};
    
\end{tikzpicture}
\end{center}

For example, imagine that the set we want to trim contains some rectangle $A_1\times B_1$. We compute the conditional probability $P(B_1\cnd A_1)=P(A_1\times B_1)/P_1(A_1)$. (Note that $P_1(A_1)=P(A_1\times\Omega_2$), so this conditional probability is the density of the rectangle in the vertical $A_1$-stripe, measured according to $P$.) We find that this conditional probability is slightly less than the threshold~$\eps$, so we let this rectangle ($A_1\times B_1$) through. Then we discover another rectangle  $A_2\times B_2$ where $A_2$ is a part of $A_1$, but $B_2$ is bigger than $B_1$ (as shown in the picture), and again find that $P(B_2\cnd A_2)$ is slightly less than~$\eps$. But if we let the second rectangle through, the average vertical measure of the resulting union may exceed $\eps$. This could happen, for example, if all the mass in $A_1\times B_1$ is concentrated outside $A_2\times B_1$; then the conditional measure of $B_1$ exceeds $\eps$ outside $A_2$ and is zero inside $A_2$, thus leaving space for additional measure from~$B_2$.

\smallskip
So the second approach also does not work. How can we deal with this problem?

\smallskip
\emph{Main idea}: We combine the two approaches and always check (before adding something) that the actual conditional probability (with the interval as the condition) is close to the tentative conditional probability computed by $\Gamma$. The latter will remain almost the same for smaller intervals (a valid computation remains valid when more information about the oracle is known), so the errors related to the change will be bounded. 

\smallskip
\emph{Details}. First we need to introduce some terminology and notation. We consider \emph{basic} (=clopen) sets in $\Omega_1\times\Omega_2$, i.e., finite unions of products of intervals. Every effectively open set is a union of a computable increasing sequence of basic sets. A basic set is a \emph{rectangle} if it is the product of two clopen sets in $\Omega$ (not necessarily intervals). \bl{By a \emph{vertical stripe} we mean a rectangle} $S=I\times \Omega_2$, where $I$ is some interval in $\Omega_1$ (i.e., $I$ consists of all extensions of some finite string). A basic set $U$ is \emph{stable in the stripe $S=I\times\Omega_2$} if $U\cap S=I\times V$ for some $V\subset\Omega_2$. This means that all the vertical sections of $U$ inside $S$ are the same; we denote these sections by $U|_S$.

The \emph{horizontal size} $h(S)$ of a stripe $S=I\times\Omega$ is defined as $P_1(I)$ (and is equal to the $P$-measure of this stripe). If a set $U$ is stable in the stripe $S$, its \emph{vertical size in $S$} is defined as $P(U\cap S)/P(S)$, i.e., the conditional probability of $U|_S$ with condition $I$. We denote the vertical size by $v(U\cnd S)$. Note that the vertical size can increase or decrease if we replace $S$ by a smaller stripe $S'$ in $S$ (and if it increases, say, for the left half of $S$, then it decreases for the right half); so ``average vertical size'' would be a better name for $v(U\cnd S)$.

\smallskip
We want to trim an effectively open set $U$ that is the union of a computable increasing sequence of basic sets 
$$
U_1\subset U_2\subset U_3\subset\ldots
$$
Let us explain first which parts of $U_1$ will be let through. We divide $\Omega_1$ into two stripes, then divide each stripe into two halves, and so on. We use the algorithm $\Gamma$ to get the approximations for the tentative conditional probabilities for all stripes. Let us agree, for example, that for stripes $S$ of level $n$ (with footprints of length $2^{-n}$) we always make $n$ steps of the $\Gamma$-computation, using $n$ first bits of the oracle $\alpha_1$ (i.e., the bits that are fixed for a stripe $S$) and produce some lower and upper bounds $\underline{P}^S(V)$ and $\bbar{P}^S(V)$ for the tentative conditional probability of all intervals $V\subset\Omega_2$. Note that for a given $V$ the interval $[\underline{P}^S(V),\bbar{P}^S(V)]$ can only decrease as $S$ becomes smaller. We know that these intervals should converge to $P(V\cnd\alpha_1)$ if we decrease the size of intervals $V$ containing $\alpha_1$; for other points the convergence is not guaranteed.

As soon as a stripe becomes small enough to make $U_1$ stable in this stripe, we compute the lower and upper bounds $\underline{v}(U_1\cnd S)$ and $\bbar{v}(U_1\cnd S)$ for the vertical size $v(U_1\cnd S)$ such that the difference between the lower and upper bounds is at most $2^{-n}$ for stripes of level $n$.\footnote{%
Since $P$ is computable, we can compute $v(U_1\cnd S)$ for each $U_1$ and $S$ with arbitrary precision. The only exception is the case when $P(S)=0$; to avoid it, let us agree that we start processing stripe $S$ only after we discover that $P(S)>0$. In this way we lose all stripes with $P(S)=0$ but this does not matter since these stripes do not contain random pairs $(\alpha_1,\alpha_2)$. (Recall that our goal was to prove that $(\alpha_1,\alpha_2)$ is not random contrary to the assumption.)
  } 
Unlike for $\underline{P}$ and $\bbar{P}$, the interval $[\underline{v}(U_1\cnd S),\bbar{v}(U_1\cnd S)]$ does not necessarily decrease as $S$ becomes smaller. Still they converge to the conditional probability of the $\alpha_1$-section of $U_1$ if $S$ are decreasing stripes around $\alpha_1$ (since conditional probability is well defined for every $P_1$-random point in $\Omega_1$).

If for some stripe $S$ all the four numbers $\underline{P}^S(U_1\cnd S)$, $\bbar{P}^S(U_1\cnd S)$, $\underline{v}(U_1\cnd S)$, and $\bbar{v}(U_1\cnd S)$ are close to each other, more precisely, if all four can be covered by some interval of size $\delta_1$ (where $\delta_1$ is a small number, see below), and at the same time the upper bound $\bbar{v}(U_1\cnd S)$ is less than the threshold $\eps$ selected for trimming, we say that $S$ is \emph{$U_1$-good} and let $U_1$ through inside $S$. Note that smaller stripes may be $U_1$-good or not, but this does not matter at this stage, since $U_1$ is already let through inside $S$.

In this way we get a trimmed version $\hat{U}_1\subset U_1$. The set $\hat{U}_1$ may not be a basic set, but it is effectively open. Before going further, let us prove some properties of this construction: 
\begin{enumerate}
  \item Assume that some pair $(\beta_1,\beta_2)$ is covered by $U_1$, the conditional probability $P(\cdot\cnd\beta_1)$ is well defined and is computed by $\Gamma$ with oracle $\beta_1$, and the $\beta_1$-section of $U_1$ has conditional measure (with condition $\beta_1$) less than $\eps$. Then $(\beta_1,\beta_2)$ is covered by~$\hat{U}_1$. 
  \item The $P$-measure of the trimmed set $\hat{U}_1$ is at most $\eps$.
\end{enumerate} 

Proof of 1: Indeed, look at the smaller and smaller stripes that contain $\beta_1$. Starting from some point, $U_1$ is stable in these stripes, and the vertical size and tentative probabilities converge to some number smaller than~$\eps$. So they finally get into $\delta_1$-interval and are all less than~$\eps$. Therefore all small enough stripes containing $\beta_1$ are $U_1$-good, and the $\beta_1$-section of $U_1$ is not trimmed.

Proof of 2: In every $U_1$-good stripe $S$ the vertical size $v(U_1\cnd S)$ is less than $\eps$, so the measure of $U_1$ inside this stripe is at most $\eps h(S)$. \bl{We may consider only maximal $U_1$-good stripes, and the sum of their horizontal sizes is bounded by $1$.}

\medskip
Now we switch to the next set $U_2$ (we should decide which part of it should remain). We start to consider $U_2$ only inside maximal $U_1$-good stripes selected at the first stage.\footnote{Note that $U_1$-good stripe may have empty intersection with $U_1$, so this does not prevent us from adding some stripes that intersect $U_2$ but not $U_1$.} Let $S$ be one of them. We start dividing $S$ into smaller stripes; at some point  they are small enough to make both $U_1$ and $U_2$ stable. Then we start checking if they are both $U_1$-good (according to our definitions) and \emph{$U_2$-good}. The latter means that they satisfy the similar requirement for $U_2$ with smaller error tolerance $\delta_2$ (the four numbers for $U_2$ are in some $\delta_2$-interval and the upper bound for the vertical size of $U_2$ in the stripe is less than $\eps$). If we find a stripe $S'$ inside $S$ that is both $U_1$-good and $U_2$-good, then the set $U_2$ is let through inside $S'$. So finally we have (inside $S$) the set
$$
(S \cap U_1) \cup \bigcup_{S'} (S'\cap U_2)
$$
where the union is taken over stripes $S'\subset S$ that are both $U_1$- and $U_2$-good. Doing this for all maximal $U_1$-good stripes $S$, we get the trimmed version $\hat{U}_2$ of $U_2$. By construction $\hat{U}_1\subset \hat{U}_2\subset U_2$. 

\smallskip
Now the key estimate for the size of $\hat{U}_2$ inside a maximal $U_1$-good stripe $S$:
\begin{lemma}\label{lem:HayatoFirstStep}
The $P$-measure of $\hat{U}_2\cap S$ is bounded by $(\eps+2\delta_1)h(S)$.
\end{lemma}
Adding these inequalities for all maximal $U_1$-good stripes $S$ (they are disjoint), we see that the total measure of $\hat{U}_2$ is bounded by $\eps+2\delta_1$. \bl{(Note that $\hat{U}_2$, as well as $\hat{U}_1$, is contained in the union of maximal $U_1$-good stripes.)}

\begin{proof}[Proof of Lemma~\textup{\ref{lem:HayatoFirstStep}}]
The measure in question can be rewritten as
$$
P(S\cap U_1) + \sum_{S'} P(S'\cap (U_2\setminus U_1))
$$ 
(we separate points added on the first and second stages). This sum can be rewritten as
$$
h(S)v(U_1\cnd S) + \sum_{S'} h(S')[v((U_2\setminus U_1)\cnd S')]
$$
or
$$
h(S)v(U_1\cnd S) + \sum_{S'} h(S')v(U_2\cnd S')-\sum_{S'}h(S')v(U_1\cnd S')
$$
Imagine for the moment that in the last term the condition is $S$, not $S'$. Then we could combine the first and last term and get
$$
\biggl(h(S)-\sum_{S'}h(S')\biggr)v(U_1\cnd S) + \sum_{S'}h(S')v(U_2\cnd S') \eqno (*)
$$
The factors $v(U_1\cnd S)$ and $v(U_2\cnd S')$ are bounded by $\eps$ (for all $S'$ where $U_2$ is let through), and the sum of horizontal sizes is just $h(S)$, so the lemma is proven without $2\delta_1$-term. This term comes because of the replacement we made: the difference between $v(U_1\cnd S)$ and $v(U_1\cnd S')$ is bounded by $2\delta_1$, and the sum of all $h(S')$ is at most $h(S)$. Indeed, the interval between lower and upper approximations $\underline{P}$, $\bbar P$ only decreases, and both sizes $v(\cdot\cnd S)$ and $v(\cdot\cnd S')$ are in $\delta_1$-neighborhood of every point in the smaller interval (that corresponds to $S'$).
\end{proof}

\bl{%
To simplify the accounting in the future, we can rewrite the bound we have proved. The second term in $(*)$ is the size of $U_2$ inside $U_1$-$U_2$-good stripes $S'$, while the first term plus the error term bounded by $2\delta_1 h(S)$ is the bound for the size of $U_1$ inside $U_1$-stripe $S$ minus $U_1$-$U_2$-good stripes.

We can add these bounds for all maximal $U_1$-good stripes. Let $G_1$ be their union, and let $G_2$ be the union of maximal $U_1$-$U_2$-good stripes (so $G_2\subset G_1$). Then $\hat{U}_2$ is empty outside $G_1$, coincides with $U_1$ inside $G_1\setminus G_2$, and coincides with $U_2$ inside $G_2$. The bounds for $\hat{U}_2$ in the last two cases are $\eps P(G_1\setminus G_2)+2\delta_1 h(G_2)$ and $\eps P(G_2)$ respectively.
}%

\smallskip
Another thing we need to check is the following. Assume that (1)~$(\beta_1,\beta_2)$ is covered by $U_2$; (2)~conditional probability $P(\cdot\cnd\beta_1)$ with condition $\beta_1$ is well defined and $\beta_1$-conditional size of the $\beta_1$-section of $U_2$ is computed by $\Gamma$ with oracle $\beta_1$; (3)~this size is less than $\eps$. Then $(\beta_1,\beta_2)$ is covered by the trimmed set.  Indeed, consider smaller and smaller stripes containing $\beta_1$. Starting from some point, both sets $U_1$ and $U_2$ are stable in those stripes,  all the approximations converge (both for $U_1$ and $U_2$), and the limits are less than $\eps$. So at some stage the stripes become both $U_1$- and $U_2$-good, and at this moment $(\beta_1,\beta_2)$ is covered (unless this happened earlier).

\medskip
Now we consider the next set $U_3$. The same construction is used: consider maximal $U_1$-$U_2$-good stripes selected at the second stage. For each of them we look for stripes inside that are both $U_2$-good and \emph{$U_3$-good} (the latter means that the set $U_3$ is stable, all four parameters are $\delta_3$-close, and the upper bound for the vertical size is less than~$\eps$). Then we do the same thing as before, but not for $U_1$-$U_2$-good stripes inside some $U_1$-good one, but for $U_2$-$U_3$-good stripes inside some $U_2$-good one. The same approach is used for $U_4$, $U_5$, etc.

\rd{%
I made one more attempt of rewriting. The inductive statement was a bit confusing, since we did not say clearly where we use the inductive assumption, and also it was not consistent with lemma before. So I introduced the notation for $G_i$ and reformulate the statement of the previous lemma in the form where it can be generalized to all layers with the same proof (the general statement is now called Lemma~\ref{lem:HayatoGeneral}), and then the result is obtained just by taking a sum, so no need for explicit induction.%
}

\bl{%
In this way we get the set $G_3$ that is the union of maximal $U_2$-$U_3$ good stripes. In this set $U_3$ is let through to be included into $\hat{U}_3$. Then we get $G_4$ where $U_4$ is let through to be included into $\hat{U}_4$, etc. The same reasoning as in the proof of Lemma~\ref{lem:HayatoFirstStep} gives us the following bounds:
\begin{lemma}\label{lem:HayatoGeneral}\leavevmode
\begin{itemize}
\item $P(U_{i-1} \cap (G_{i-1}\setminus G_i))\le \eps P(G_{i-1}\setminus G_i)+2\delta_{i-1} P(G_i)$\textup;
\item $P(U_i\cap G_i)\le \eps P(G_i)$.
\end{itemize}
\end{lemma}

}

What have we achieved? We explained how to trim the set  $U_k$ for each $k$ and get $\hat{U}_k\subset U_k$. The union $\hat U=\bigcup_k \bl{\hat{U_k}}$ is the trimmed version of the effectively open set $U$ we started with.  \bl{In other words, $\hat{U}$ coincides with $U_{i-1}$ inside $G_{i-1}\setminus G_i$ and coincides with $U$ in $\cap_i G_i$}. What are the properties of this $\hat U$? 
\begin{itemize}
\item The trimming procedure is effective: the set $\hat U$ is effectively open uniformly in $U$. This is guaranteed by the construction.

\item The $P$-measure of $\hat U$ is small. Indeed, for each $k$ the measure of $\hat{U}_k$ is bounded by $\eps+2\sum_i \delta_i$ \bl{(sum of the first bounds from Lemma~\ref{lem:HayatoGeneral} for $i=2,\ldots,k$ and the second bound for $i=k$). Then we note that computable $\delta_i$ can be chosen in such a way that $\sum\delta_i<\eps$, and we achieve $P(\hat U)\le 3\eps$, since $P(G_i)\le 1$ for all $i$.} 

\item Assume that conditional probability is well defined for some condition $\beta_1$ and is computed by $\Gamma$ with oracle $\beta_1$. Assume also that the $\beta_1$-section of $U$ has conditional measure less than $\eps$ and contains some $\beta_2$. Then $(\beta_1,\beta_2)\in\hat U$. Indeed, $(\beta_1,\beta_2)$ belongs to some $U_i$ and (under the conditions mentioned) belongs to $\hat{U_i}$ as explained above.
\end{itemize}

Then the proof ends in the same way as in the standard van Lambalgen theorem: since $(\alpha_1,\alpha_2)$ is $P$-random, the first coordinate $\alpha_1$ is $P_1$-random, the conditional probability $P(\cdot\cnd\alpha_1)$ is well defined, and our assumption says that it is computed by $\Gamma$ with oracle $\alpha_1$. If $\alpha_2$ is not (Martin-L\"of) random with respect to the conditional probability, we consider a Martin-L\"of test with oracle $\alpha_1$ for measure $P(\cdot\cnd\alpha_1)$ rejecting $\alpha_2$, represent its elements as $\alpha_1$-sections of a sequence of uniformly effectively open sets $U_i\subset\Omega_1\times\Omega_2$, and trim these sets. This gives us a Martin-L\"of test with respect to measure $P$ that rejects $(\alpha_1,\alpha_2)$, which contradicts the assumption.
\end{proof}

\section{Quantitative version for uniformly computable\\ conditional probabilities}

In the previous section we started with a computable probability distribution $P$ on $\Omega_1\times\Omega_2$, and then defined the conditional distributions on $\Omega_2$. However, in many cases the natural order could be different: we first generate a sequence $\omega$ randomly according to some distribution $P_1$ on $\Omega_1$, and then generate $\omega'$ randomly according to some distribution $P^{\omega}$ on $\Omega_2$ that depends on $\omega$. If the dependence of $P^\omega$ on $\omega$ is computable, then we get some computable distribution $P$ on $\Omega_1\times\Omega_2$. It is easy to check that for $P$ the conditional probabilities indeed coincide with $P^\omega$, so we can apply Takahashi's results from the preceding section. But in this special case the argument could be easier, and a stronger quantitative version could be obtained (as shown by Vovk and Vyugin in \cite[Theorem 1, page 261]{VovkBayesian}, though in somehow obscure notation).

To state this quantitative version, we need to use the notion of randomness deficiency. More precisely, we use expectation-bounded randomness deficiency (see~\cite{survey} for the details). In other words, we consider the maximal (up to $O(1)$-factor) lower semicomputable function $t$ on $\Omega_1$ with non-negative real values (including $+\infty$) such that 
  $$
\int_{\Omega_1} t(\omega)\,dP_1(\omega)\le 1.  
  $$
One can prove (see~\cite{survey}) that such a function exists. We denote this maximal  function by $\ttt_{P_1}(\omega)$; the value $\ttt_{P_1}(\omega)$ is finite for $P_1$-random $\omega$ and infinite for non-random ones. Then we switch to the logarithmic scale and define deficiency as $\dd_{P_1}(\omega)=\log \ttt_{P_1}(\omega)$.

In a similar way one can define randomness deficiency for pairs with respect to $P$: it is the logarithm of the maximal lower semicomputable function $t(\omega,\omega')$ on $\Omega_1\times\Omega_2$ such that
  $$ 
\iint_{\Omega_1\times\Omega_2} t(\omega,\omega')\,dP(\omega,\omega')=\int_{\omega} \int_{\omega'} t(\omega,\omega')\,dP^\omega(\omega')\,dP_1(\omega)\le 1.
  $$
We denote this maximal function by $\ttt_P(\omega,\omega')$ and its logarithm by $\dd_P(\omega,\omega')$.  

We need one more variant of randomness deficiency, and it is a bit more complicated.  We want to measure the randomness deficiency of $\omega'$ with respect to the measure $P^\omega$ given some additional information as oracle. This additional information is $\omega$ itself and some integer (its role will be explained later). We can use the general definition of uniform deficiency (as a function of a sequence and a measure, see~\cite{survey}), but let us give an equivalent definition for this special case. A lower semicomputable function $t(\omega',\omega,k)$ of three arguments ($\omega'$ and $\omega$ are sequences, $k$ is an integer) is called a test, if 
$$
   \int_{\omega'} t(\omega',\omega,k) \,dP^{\omega}(\omega')\le 1 
$$
for every $\omega$ and $k$. There exists a maximal test, as usual: we may trim all the lower semicomputable functions making them tests, and then take their sum with coefficients $1/2^n$ (or other converging series). Trimming is easy since $\omega$ is an argument and $P^\omega$ is computable given $\omega$ (uniformly for all $\omega$, according to our assumption). We denote the maximal test by $\ttt_{P^\omega}(\omega'\cnd \omega,k)$ and its logarithm by $\dd_{P^\omega}(\omega'\cnd \omega,k)$. (We use this notation since $\omega$ is a parameter and $k$ is an additional condition.)

Now we can state the Vovk--Vyugin result:

\begin{theorem}\label{th:vv}
$$
\dd_P(\omega,\omega')=\dd_{P_1}(\omega)+\dd_{P^\omega}(\omega'\cnd \omega, \dd_{P_1}(\omega))+O(1).
$$
\end{theorem}

In this statement we assume that the value of $\dd_{P_1}(\omega)$ in the condition is rounded to an integer; the exact nature of rounding does not matter since it changes the deficiency only by $O(1)$.

\smallskip
Again, before proving this theorem, let us make some remarks:
\begin{itemize}
\item This result has high precision (up to $O(1)$ additive term); if we were satisfied with logarithmic precision, we could omit $\dd_{P_1} (\omega)$ in the condition. Indeed, the standard argument shows that adding condition $d$ could increase the deficiency at most by $O(\log d)$ and decrease it at most by $O(1)$.

\item It is easy to see that $\dd_{P^\omega}(\omega'\cnd\omega,d)$ is finite if and only if $\omega'$ is random with respect to measure $P^\omega$ with oracle $\omega$ (as we have mentioned, adding the condition $d$ changes the deficiency at most by $O(\log d)$, so [in]finite values remain [in]finite). So we get a qualitative version: $(\omega,\omega')$ is $P$-random if and only if $\omega$ is $P_1$-random and $\omega'$ is  $P^\omega$-random with oracle $\omega$. (This statement generalizes the van Lambalgen theorem and is a special case of Takahashi's result considered in the previous section.)

\item
A special case of this statement, when $P^\omega$ does not depend on $\omega$ and is always equal to some computable measure $P_2$, gives a quantitative version of the van Lambalgen theorem for product measure $P_1\times P_2$.

\item
One can consider a finite version of this theorem. If $x$ is a constructive object, e.g., a string or a pair of strings, and $A$ is a finite set containing $x$, we may define the randomness deficiency of $x$ as an element of $A$ in the following way:
$$
d(x\cnd A) = \log|A|-\KP(x\cnd A),
$$
where $\KP(x\cnd A)$ is the conditional prefix complexity of $x$ given $A$. It is easy to check that $d(x\cnd A)$ is positive (up to $O(1)$-error) and that it can also be defined as a logarithm of maximal lower semicomputable function $t(x,A)$ of two arguments ($x$ is an object, $A$ is a finite set) such that 
$$
   \sum_{x\in A} t(x,A) \le 1
$$ 
for each finite set $A$.  We can also define randomness deficiency with an additional condition as
$$
 d(x\cnd A; y)=\log|A|-\KP(x\cnd A,y).
$$
Then we can state the following equality for the deficiency of \bl{a} pair (with $O(1)$-precision):
$$
d((x,y)\cnd A\times B)=d(x\cnd A; B)+d(y\cnd B; x, A, d(x\cnd A; B)).
$$
It is just the Levin--Gacs formula for the complexity of pairs in disguise. Indeed, this statement can be rewritten (with $O(1)$-precision) as
    \begin{multline*}
\log|A\times B|-\KP(x,y\cnd A,B)=\\
=\log|A|-\KP(x\cnd A,B)+
\log|B|-\KP(y\cnd B,x,A,\log|A|-\KP(x\cnd A,B)).
    \end{multline*}
The logarithms cancel each other and we have
$$
\KP(x,y\cnd A,B)=
\KP(x\cnd A,B)+\KP(y\cnd B,x,A,\log|A|-\KP(x\cnd A,B)).
$$
which is just the Levin--Gacs theorem about prefix complexity of a pair (note that $\log |A|$ in the condition does not matter since $A$ is there anyway). 

The proof below can be also adapted to the finite case.\footnote{It would be interesting to derive the statement of the infinite theorem using the formula for expectation-bounded deficiency in terms of prefix complexity and the formula for the complexity of pairs, but it is not clear how (and if) this can be done.}
\end{itemize}

\begin{proof}[Proof of Theorem~\textup{\ref{th:vv}}]
We need to prove two inequalities. In each case, we construct some test and compare it with the maximal one. 

We start with the $\ge$-direction, proving that $d(\omega,\omega')$ is large \bl{enough}. The function
$$
T(\omega,\omega')=\ttt_{P_1}(\omega)\cdot \ttt_{P^\omega}(\omega'\cnd \omega, \dd_{P_1}(\omega))
$$
(where $\dd_{P_1}(\omega)$ in the argument is rounded) has integral at most $1$ with respect to measure~$P$. Indeed, 
  \begin{multline*}
\int_\omega\int_{\omega'}\ttt_{P_1}(\omega)\cdot \ttt_{P^\omega}(\omega'\cnd \omega, \dd_{P_1}(\omega))\,dP^\omega(\omega')\,dP_1(\omega)=\\=
\int_\omega\ttt_{P_1}(\omega)\int_{\omega'} \ttt_{P^\omega}(\omega'\cnd \omega, \dd_{P_1}(\omega))\,dP^\omega(\omega')\,dP_1(\omega)\le
\int_\omega\ttt_{P_1}(\omega)\,dP_1(\omega)\le 1
  \end{multline*}
(first we use that $\ttt_{P^\omega}$ is a test, and then we use that $\ttt_{P_1}$ is a test).  One would like to say that this test $T$ is bounded by the maximal test $\ttt_P$, but the problem is that the function $T$ is not guaranteed to be a test: it may not be lower semicomputable, since it uses $\dd_{P_1}(\omega)$ as a condition. To avoid this problem, we consider a bigger function
$$
T'(\omega,\omega')=\sum_{k<\dd_{P_1}(\omega)} 2^{k}\ttt_{P^\omega}(\omega'\cnd\omega,k).
$$
It is indeed bigger (up to $O(1)$-factor) since the last term in the sum coincides with $T$ up to $O(1)$-factor. This function is lower semicomputable since the property $k<\dd_{P_1}(\omega)$ is effectively open in the natural sense, and $\ttt_{P^\omega}$ is lower semicomputable. And the integral is bounded not only for $T$ but also for $T'$:
\begin{multline*}
\iint\biggl[\sum_{k<\dd_{P_1}(\omega)} 2^{k}\ttt_{P^\omega}(\omega'\cnd\omega,k)\biggr]\,dP^\omega(\omega')\,dP_1(\omega) =\\
=
\int_\omega \biggl[\sum_{k<\dd_{P_1}(\omega)} 2^k \int_{\omega'}\ttt_{P^\omega}(\omega'\cnd\omega,k)\,dP^\omega(\omega')\biggr]\,dP_1(\omega) \le\int_\omega \biggl[\sum_{k<\dd_{P_1}(\omega)} 2^k\biggr]\,dP_1(\omega)\le\\\le O(1)\cdot \int_\omega 2^{\dd_{P_1}(\omega)}\,dP_1(\omega)=O(1)\cdot\int_\omega \ttt_{P_1}(\omega)\,dP_1(\omega)=O(1).
\end{multline*}
Here we use that the sum of different powers of $2$ coincides with its biggest term up to $O(1)$-factor. So we have found a function $T'$ that is lower semicomputable and is a test, so $T'$ and therefore $T$ are bounded by the maximal $P$-test, which gives the required inequality.

Now we have to prove the reversed ($\le$) inequality. For that we consider the maximal $P$-test $\ttt_P(\omega,\omega')$ and the maximal $P_1$-test $\ttt_{P_1}(\omega)$. The function
$$
\omega\mapsto\int_{\omega'} \ttt_P(\omega,\omega')\,dP^{\omega}(\omega')
$$
is lower semicomputable and its integral with respect to measure $P_1$ is at most $1$, therefore this function is bounded by $O(1)\cdot \ttt_{P_1}(\omega)$. So the ratio 
 $$
t(\omega',\omega)=\ttt_P(\omega,\omega')/\ttt_{P_1}(\omega)
 $$ 
has bounded integral over $\omega'$ (with respect to $P^\omega$) for each $\omega$ (and the bound does not depend on~$\omega$). If $t$ were a test, we could compare $t$ with the maximal test $\ttt_{P^\omega}(\omega'\cnd\omega)$ and get the desired inequality. But again the function $t$ may not be lower semicomputable, since it has lower semicomputable function in the denominator.

This is why we need an additional argument $d$ for the test function. Namely, we consider
 $$
t(\omega'\cnd \omega,d)=[\ttt_P(\omega,\omega')/2^{d+c}],
 $$ 
where the square brackets denote trimming that make this function a test (the integral over $\omega'$ for each $\omega$ and $d$ should be bounded by $1$, and the trimming should not change $t(\omega'\cnd \omega,d)$ if for this $\omega$ the integral was already bounded by $1$). The constant $c$ should be chosen in such a way that for $d=\dd_{P_1}(\omega)$ trimming is not needed; this is possible due to the argument above.

It remains to compare this test with the maximal one and note that for this test the required inequality is true by construction.
\end{proof}

There is another generalization of the van Lambalgen theorem. In the previous results we considered only computable measures. However, one can define \emph{uniform randomness test} as a lower semicomputable function $t(\omega,P)$ of two arguments (where $\omega$ is a sequence, and $P$ is a probability distribution on the Cantor space) such that
$$
\int_{\Omega} t(\omega,P)\,dP(\omega)\le 1
$$
for every $P$. Note that we should first define the notion of semicomputability for functions whose arguments are measures; this can be done in a natural way (even for arguments in an arbitrary constructive metric space, see~\cite{survey} for the details). Also we can generalize this definition by allowing points in constructive metric spaces as additional conditions. After that one could prove that
$$
\dd((\omega_1,\omega_2)\cnd P_1\times P_2)=
\dd(\omega_1\cnd P_1; P_2)+
\dd(\omega_2\cnd P_2; P_1,\omega_1,\dd(\omega_1\cnd P_1;P_2)).
$$
if all these quantities are defined in a natural way. It would be interesting to combine this generalization with Vovk--Vyugin result (where $P_2$ is not fixed, but depends on $\omega_1$). One possibility is to assume that $P_1$ is a computable function of some parameter $p$ (a point in a constructive metric space). For example, we may let $P_1=p$, so $P_1$ itself may be used as such a parameter. Then we assume that $P_2^\omega(\cdot)$ is a computable function of $p$ and $\omega$, so the distribution on pairs also becomes a computable function of $p$. In this case we get the equality
 $$
\dd((\omega_1,\omega_2)\cnd P_1; p)=
\dd(\omega_1\cnd P_1; p)+
\dd(\omega_2\cnd P_2^{\omega_1}; p,\omega_1,\dd(\omega_1\cnd P_1;p))+O(1),
$$
where $O(1)$-constant does not depend on $p$.

\section{Acknowledgements} 

\bl{%
We are grateful to the organizers of the ``Focus Semester on Algorithmic
Randomness'' in June 2015:  Klaus Ambos-Spies, Anja Kamp, Nadine Losert, Wolfgang Merkle, and Martin Monath. We thank the Heidelberg university and Templeton foundation for financial support.  

Alexander Shen thanks Vitaly Arzumanyan, Alexey Chernov, Andrei Romashchenko, Nikolay Vereshchagin, and all members of Kolmogorov seminar group in Moscow and ESCAPE team in Montpellier.
}%

\end{document}